\documentclass[11pt]{amsart}
\usepackage{geometry}                
\usepackage{graphicx}
\usepackage{amssymb,amsmath}
\usepackage{epstopdf}
\usepackage{tikz}
\usepackage{lscape}
\usepackage{color}
\usepackage{todonotes}
\usepackage{mathrsfs}
\usepackage{longtable}

\theoremstyle{plain}
\newtheorem{theorem}{Theorem}[section]

\newtheorem{corollary}[theorem]{Corollary}
\newtheorem{proposition}[theorem]{Proposition}

\theoremstyle{definition}
\newtheorem{definition}[theorem]{Definition}
\newtheorem{example}[theorem]{Example}
\newtheorem{remark}[theorem]{Remark}

\newcommand{\A}{\mathcal{A}}
\newcommand{\B}{\mathcal{B}}

\newcommand{\Z}{\mathbb{Z}}

\newcommand{\rk}{\mathrm{rk}}
\newcommand{\calP}{\mathcal{P}}

\newcommand{\Fl}{\mathcal{F}l}

\newcommand{\gc}{\ensuremath{\chi}}
\newcommand{\gd}{\ensuremath{\delta}}

\newcommand{\gf}{\ensuremath{\varphi}}

\newcommand{\gm}{\ensuremath{\mu}}

\newcommand{\gz}{\ensuremath{\zeta}}

\title{Partial flag incidence algebras}
\author{Max Wakefield}
\thanks{Supported by the Simons Foundation, the Office of Naval Research, and the Japanese Society for the Promotion of Science.}
\address{US Naval Academy\\
  572-C Holloway Rd\\
  Annapolis MD, 21402 USA}
\email[]{wakefiel@usna.edu}

\begin{document}


\maketitle

\begin{abstract} The $n^{th}$ partial flag incidence algebra of a poset $P$ is the set of functions from $P^n$ to some ring which are zero on non-partial flag vectors. These partial flag incidence algebras for $n>2$ are not commutative, not unitary, and not associative. However, partial flag incidence algebras contain generalized zeta, delta, and M\"obius functions which are finer and more delicate invariants than their classical analogues.  We also study some generalized characteristic polynomials of posets which are not evaluations of Tutte polynomials and compute them for Boolean lattices. Motivation for this work came from studying the matroid Kazhdan-Lusztig polynomials where partial flag Whitney numbers play a central role.\end{abstract}

\vspace{10pt}

\section{Introduction}

The classical incidence algebras are functions from the set of intervals in a poset to some ring. These algebras contain some of the most studied invariants for posets, matroids, and graphs like the chromatic polynomial and the M\"obius function. They are an important source of associative, unital, but non-commutative algebras. In this note we define and study a generalization of the classic incidence algebras. The algebras we study are functions from the set of ordered sequences (aka flags) in a poset to some ring. These algebras contain many interesting invariants including generalizations of the classical M\"obius functions and characteristic polynomials. 

The product we use for these algebras is non-trivial. We modeled it after both the product in the classical incidence algebras and the higher dimensional integral convolution product in functional analysis. This product is complicated and results in these algebras being non-associative and non-unital. At a first glance this seems to be bad news. On the other hand as long as the coefficient ring is commutative they do satisfy a product/tensor or K\"unneth formula. Without this formula computations on infinite families of posets would be very difficult. 

There are quite a few generalizations of incidence algebras, see \cite{MR14,V07,R86,HS89}. Also, there are studies of certain subalgebras of the classical incidence algebras, for example the finitary incidence algebras studied in \cite{KKW-19}. The algebras considered here seem to be new. Koszulness of incidence algebras was studied in \cite{RS10}. The classical incidence algebras have a nice Hopf algebra structure (see \cite{Dur86}). Another important result for incidence algebras is the classification of reduced incidence algebras in \cite{Kr81}. Recently Khrypchenko studied the algebras defined here and has shown that they essentially contain all the information of the underlying poset (see \cite{khryp-21}). These results warrant future investigations into these properties for these generalized incidence algebras.

A central reason for this study is to lay the foundation for studying various generalized invariants for posets, lattices, and matroids. All of Section \ref{whitneysec} is devoted to studying generalized M\"obius functions. There we show that these generalized M\"obius functions satisfy generalizations of some of the results for classical M\"obius functions. In particular, in Theorem \ref{cross-cut}, we show that one of our flag M\"obius functions satisfies a generalized version of Rota's Cross-cut Theorem (see \cite{rota64}). One aim is to develop these invariants for future studies on generalized Tutte polynomials.

To conclude this note we study a partial flag version of the characteristic polynomial. If the poset is the graph partitions poset of a graph then this gives a generalized chromatic polynomial. So, far this polynomial seems to be very mysterious. First we compute these polynomials for the Boolean lattice. Then we use this to demonstrate that in general these polynomials will not satisfy any kind of deletion-contraction formula. Hence these polynomials are new invariants and not evaluations of Tutte polynomials. Various generalizations of characteristic polynomials have been considered in the past. For example those studied in \cite{Chen10}, but these are evaluations of Tutte polynomials.

Another motivation for this work comes from computing the matroid Kazhdan-Lusztig polynomials defined in \cite{EPW15}. In computing these polynomials one is naturally led to counting partial flags in posets. This was further refined in \cite{PXY,W17}. It turns out that these polynomials give the Betti numbers for the intersection cohomology of the reciprocal plane if the matroid was realizable. For non-realizable matroids it was proved by Braden, Huh, Matherne, Proudfoot, and Wang in \cite{BHMPW-20-1} and \cite{BHMPW-20-2} building off of work by Huh and Wang \cite{HW17} that the coefficients of these polynomials is always non-negative. The flag Whitney number formula in \cite{W17} together with the non-negativity result of Braden, Huh, Matherne, Proudfoot, and Wang shows some very interesting properties of a matroid. One aim of this work is to find more interesting inequalities concerning number of certain types of flats of matroids.

This article is organized as follows. In Section \ref{fiadef} we present the main definition of a partial flag incidence algebra and certain canonical special elements. Then we use these elements to show that these algebras are not necessarily associative or unital. In Section \ref{whitneysec} we define some generalized M\"obius functions and flag Whitney numbers of the first and second kind and compute them for the family of Boolean posets. There we also examine these generalized M\"obius functions in the light of the classical results on the classical M\"obius unctions, in particular Rota's cross cut theorem. In Section \ref{Char_k} we define a flag version of the classic characteristic polynomial, compute it for the Boolean posets, and then show that it is not a Tutte invariant (i.e. does not satisfy a deletion contraction formula).

\

\noindent {\bf Acknowledgments:}
The author is grateful to the University of Oregon, Centro di Ricerca Matematica Ennio De Giorgi, and Hokkaido University for their hospitality during which some of this work was conducted. For helpful contributions the author would also like to thank Will Traves, David Joyner, David Phillips, Graham Denham, Michael Falk, Sergey Yuzvinsky, Hal Schenck, Takuro Abe, Hiroaki Terao, Masahiko Yoshinaga, June Huh, and Emanuele Delucchi.

\section{Flag incidence algebras}\label{fiadef}

For the definition of incidence algebras we need to restrict to the class of locally finite posets $\calP$. This means that any interval $[a,b]=\{x\in \calP|a\leq x\leq b\}$ is a finite set. Here is the main definition of this note.

\begin{definition}

Let $\calP$ be a locally finite poset. The \emph{$n^{th}$ flag incidence algebra} on $\calP$, for $n\geq 2$, with coefficients in a ring $R$ is $\mathcal{I}^n(\calP,R)$ (we will suppress the $R$ when it is clear from context) the set of all functions $f:\Fl^n(\calP)\to R$ where $$\Fl^n(\calP)=\{(X_1,\dots ,X_n)\in \calP^n| X_1\leq X_2\leq \cdots \leq X_n\}.$$ Addition in $\mathcal{I}^n(\calP)$ is given by $$(f+g)(X_1,\dots ,X_n)=f(X_1,\dots ,X_n)+g(X_1,\dots ,X_n)$$ and multiplication is given by a convolution $$(f*g)(X_1,\dots ,X_n)=\sum\limits_{X_i\leq Y_i\leq X_{i+1}}\hspace{-.5cm}f(X_1,Y_1,Y_2,\dots ,Y_{n-1})g(Y_1,Y_2,\dots ,Y_{n-1},X_n) $$ where the juxtaposition above is just multiplication in the ring $R$.

\end{definition}

The usual incidence algebra is the case of $n=2$. The multiplication in this definition is in some sense the most important part of the paper. The idea is to generalize the usual incidence algebra to flags with length longer than 2 and in some way mimic the classical convolution in functional analysis as an integral. Using this multiplication leads to two major problems:

\begin{enumerate}

\item The algebras in general are not associative (see Proposition \ref{notassoc} below).

\item The algebras in general are not unital (see Proposition \ref{notunital} below).

\end{enumerate}

However, the aim of this note is to show that these algebras still have some worth when studying posets. In particular, we can easily generalize some classical invariants which contain richer and finer information. First we define some particularly important elements of $\mathcal{I}^n(\calP)$ in order to study the basic algebraic structure of the algebras $\mathcal{I}^n(\calP)$.

\begin{definition}

\begin{enumerate}

\item For an ordered subset $I=\{i_1,\dots ,i_s\}\subseteq [n]$ the \emph{piece-wise delta function} is $\gd_I\in \mathcal{I}^n(\calP)$ defined by $$\gd_I(X_1,\dots ,X_n)=\begin{cases} 1 &  X_{i_1}=X_{i_2}=\cdots =X_{i_s}\\
0 &else\end{cases}.$$ (We will usually suppress the set distinguishing brackets on $I$.)

\item For subset of flags $S\subseteq \Fl^n(\calP)$ the \emph{characteristic function} $C_S\in \mathcal{I}^n(\calP)$ with respect to $S$ is defined by $$C_S(X_1,\dots ,X_n)=\begin{cases}1 & (X_1,\dots ,X_n)\in S\\
0& else\\
\end{cases}.$$

\item The $k^{th}$-\emph{zeta} function $\gz_k$ on $\calP$ is the constant function 1 on $\Fl^n(\calP)$, so for all $(X_1,\dots ,X_k)\in \Fl^n(\calP)$ $$\gz_k (X_1,\dots ,X_k)=1 . $$

\end{enumerate}
\end{definition}

\begin{remark}

Note that we can think of $\zeta_k=\gd_{1}$. And if $n=2$ then the usual delta function is $\delta=\gd_{[2]}$ in our notation.

\end{remark}

In the next section we will define the partial flag version of the classical M\"obius function. For now we can use the delta, zeta and characteristic functions to examine basic algebraic properties of the $n^th$ incidence algebras.


\begin{proposition}\label{notassoc}

The incidence algebra $\mathcal{I}^n(\calP)$ is not associative for $n>2$ and non-trivial $\calP$. 

\end{proposition}

\begin{proof}

To do this we use the functions $\gd_{1,2}$, $\gd_{2,3}$, and $\gz:=\gz_n$. First we compute \begin{multline}\label{ri1}((\gd_{1,2}*\gd_{2,3})*\gz )(X_1,\dots ,X_n)=\sum\limits_{X_i\leq Y_i\leq X_{i+1}}\gd_{1,2}*\gd_{2,3}(X_1,Y_1,\dots ,Y_{n-1})\\ =\sum\limits_{X_i\leq Y_i\leq X_{i+1}}\Bigg[\sum\limits_{Y_{i-1}\leq Z_i\leq Y_{i}}\gd_{1,2}(X_1,Z_1,Z_2,\dots ,Z_{n-1})\gd_{2,3}(Z_1,Z_2,Z_3,\dots ,Z_{n-1},Y_{n-1})\Bigg]\\ =\sum\limits_{X_i\leq Y_i\leq X_{i+1}}|[Y_3,Y_4,\dots ,Y_{n-1}]|\end{multline} where $[Y_3,\dots ,Y_{n-1}]=\{ [s_1,\dots ,s_{n-3}]\in \Fl^{n-3}| Y_{i+2}\leq s_i\leq Y_{i+3}\}$. Now we compute \begin{multline}\label{ri2} (\gd_{1,2}*(\gd_{2,3}*\gz ))(X_1,\dots ,X_n)=\sum\limits_{X_i\leq Y_i\leq X_{i+1}}(\gd_{2,3}*\gz)(X_1,Y_2,\dots ,Y_{n-1},X_n)\\
=\sum\limits_{X_i\leq Y_i\leq X_{i+1}}\left[\sum\limits_{\substack{Y_{i-1}\leq Z_i\leq Y_{i}\\ Y_{n-1}\leq Z_{n-1}\leq X_n}}\gd_{2,3}(X_1,Z_1,\dots ,Z_{n-1})\right] \\
=\sum\limits_{X_i\leq Y_i\leq X_{i+1}}|[Y_3,\dots ,Y_{n-1},X_n]| .\end{multline} Note that the element $\gd_{2,3}$ is not even defined unless $n>2$. As long as $X_{n-1}<X_{n}$ we have that (\ref{ri1}) is strictly less than (\ref{ri2}) for $n>2$ and this finishes the proof.\end{proof}

Next we will prove that $\mathcal{I}^n$ is not unital for $n>2$. 

\begin{proposition}\label{notunital}

The incidence algebra $\mathcal{I}^n(\calP)$ does not have a one sided unit for $n>2$ and non-trivial $\calP$. 

\end{proposition}

\begin{proof}

Suppose there was a right unit $u\in \mathcal{I}^n(\calP)$. Compute \begin{equation}\label{ud}(\gd_{1,n}*u)(X_1,\dots ,X_n)=u(X_1,X_1,X_1,\dots ,X_1,X_n)=\begin{cases}1 & X_1=X_n\\
0 & else\end{cases} \end{equation} and \begin{equation}\label{ud2}(\gd_{n-2,n-1}*u)(X_1,\dots ,X_1,X_n ,X_n)=u(X_1,X_1,\dots ,X_1,X_n,X_n)=\begin{cases}1 & X_1=X_n\\
0 & else\end{cases} .\end{equation}
Note that the last expression (\ref{ud2}) is only defined when $n>2$. Also, suppose that $X_n$ covers $X_1$, so that $X_1<X_n$ and there does not exist $Y$ such that $X_1<Y<X_n$. Then both expressions (\ref{ud}) and (\ref{ud2}) are zero. But $$\gz *u(X_1,\dots ,X_1,X_1,X_n)=\sum\limits_{X_1\leq Y\leq X_n}u(X_1,\dots ,X_1,Y,X_n)$$ $$=u(X_1,\dots ,X_1,X_n)+u(X_1,\dots ,X_1,X_n,X_n)=1$$ which is a contradiction. The left side is similar it just uses the computations \\
$u*\gd_{1,n}(X_1,\dots ,X_n)$, $u*\gd_{2,3}(X_1,X_1,X_n,\dots ,X_n)$, and $u*\gz (X_1,X_n,\dots ,X_n)$.\end{proof}

So far this all seems to be bad news for these higher dimensional incidence algebras. But we do have a very natural product formula generalizing the usual formula (see Proposition 2.1.12 in \cite{SO97}).

\begin{theorem}\label{prodI}

If $P$ and $Q$ are finite posets and the coefficient ring $R$ is commutative then $$\mathcal{I}^n(P\times Q,R)\cong \mathcal{I}^n(P,R)\otimes_R \mathcal{I}^n(Q,R).$$

\end{theorem}

\begin{proof}

We define a map $\gf : \mathcal{I}^n(P,R)\otimes_R \mathcal{I}^n(Q,R)\to \mathcal{I}^n(P\times Q,R)$ on simple tensors by $$\gf (g\otimes h) ((X_1,Y_1),\dots ,(X_n,Y_n))=g(X_1,\dots ,X_n)h(Y_1,\dots ,Y_n) $$ where the product on the right hand side is just multiplication in the ring $R$. First we show that $\gf$ is a ring homomorphism. To shorten the notation let $(\bar{X},\bar{Y}):=((X_1,Y_1),\dots ,(X_n,Y_n))$. On simple tensors \begin{align*}\gf((e\otimes f)\cdot (g\otimes h))(\bar{X},\bar{Y}) &=\gf (e*_Pg\otimes f*_Qh)(\bar{X},\bar{Y})\\
&=(e*_Pg)(\bar{X},\bar{Y})(f*_Qh)(\bar{X},\bar{Y})\\
&=\Bigg( \sum\limits_{\bar{X}\leq \bar{W}}e(\bar{W})g(\bar{W})\Bigg)\Bigg(\sum\limits_{\bar{Y}\leq \bar{Z}}f(\bar{Z})h(\bar{Z})\Bigg)\\
&=\sum\limits_{\bar{X}\leq \bar{W}}\sum\limits_{\bar{Y}\leq \bar{Z}}e(\bar{W})f(\bar{Z})g(\bar{W})h(\bar{Z})\\
&= \sum\limits_{\bar{X}\leq \bar{W}} \sum\limits_{\bar{Y}\leq \bar{Z}}\gf(e\otimes f)(\bar{W},\bar{Z})\gf(g\otimes h)(\bar{W},\bar{Z})\\
&=\big( \gf (e\otimes f )*_{P\times Q}\gf (g\otimes h)\big)(\bar{X},\bar{Y}) \\
\end{align*} where in the third, fourth, and fifth lines above the inequality $\bar{X}\leq \bar{W}$ means the usual $X_1\leq W_1\leq X_2\leq W_2\leq \cdots \leq W_{n-1}\leq X_n$ and $e(\bar{W})g(\bar{W})$ means the usual $$e(X_1,W_1,\dots ,W_{n-1})g(W_1,\dots ,W_{n-1},X_m) .$$ Then a routine check on non-simple tensors shows $\gf$ is a ring homomorphism.

The fact that $\gf $ is injective is nearly a tautology. Surjectivity is more interesting and there we use the finiteness hypothesis of $P$ and $Q$. Let $f\in \mathcal{I}^n(P\times Q,R)$. Then define $$F=\sum\limits_{\bar{X}\in \Fl^n(P)}\sum\limits_{\bar{Y}\in \Fl^n(Q)}f((X_1,Y_1),\dots ,(X_n,Y_n))C_{(X_1,\dots ,X_n)}\otimes C_{(Y_1,\dots ,Y_n)} $$ which is well defined since the posets are finite and $F\in  \mathcal{I}^n(P,R)\otimes_R \mathcal{I}^n(Q,R)$. Since $\gf (C_{(X_1,\dots ,X_n)}\otimes C_{(Y_1,\dots ,Y_n)})=C_{((X_1,Y_1),\dots ,(X_n,Y_n))}$ we have that $\gf(F)=f$. \end{proof}



\section{Partial flag M\"obius functions}\label{whitneysec}

Again we assume that $\calP$ is a locally finite poset. In order to save some notational space we will denote elements in $\Fl^k(\calP)$ by $\overline{X}=(X_1,\dots ,X_k)$. There is a natural partial order on $\Fl^k(\calP)$ given by $\overline{X} \preceq \overline{Y}$ if for all $1\leq i\leq k$ we have $X_i\leq Y_i$ where $\leq $ is the order in $\calP$. Then the pair $(\Fl^k(\calP), \preceq)$ is a partially ordered set. We define another kind of relation between $\Fl^{k-1}(\calP)$ and $\Fl^k(\calP)$ in order to make our definition of various partial flag M\"obius functions concise. For $\overline{Y}\in \Fl^{k-1}(\calP)$ and $\overline{X}\in \Fl^k(\calP)$ we say $\overline{Y} \trianglelefteq \overline{X}$ if and only if $$X_1\leq Y_1\leq X_2 \leq Y_2\leq X_3\leq \cdots \leq X_{k-1}\leq Y_{k-1} \leq X_k.$$ Using this we define a few different possible generalizations of the classical M\"obius function.

\begin{definition}\label{Mob-def}
The $k^{th}$ \emph{left partial flag M\"obius function} on $\calP$ is $\gm_k^l:\Fl^k(\calP)\to \Z$ recursively defined for all $\overline{X}\in \Fl^k(\calP)$ by  $$\sum\limits_{\overline{Y}\trianglelefteq \overline{X}}\gm_k^l (X_1,Y_1,\dots ,Y_{k-1})=\gd_{[k]}(\overline{X}).$$ The $k^{th}$ \emph{right partial flag M\"obius function} on $\calP$ is $\gm_k^r:\Fl^k(\calP)\to \Z$ recursively defined for all $\overline{X}\in \Fl^k(\calP)$ by  $$\sum\limits_{\overline{Y}\trianglelefteq \overline{X}}\gm_k^r (Y_1,\dots ,Y_{k-1},X_k)=\gd_{[k]}(\overline{X}).$$ The $k^{th}$ \emph{product partial flag M\"obius function} on $\calP$ is $\gm_k^p:\Fl^k(\calP)\to \Z$ defined by $$\gm_k^p (\overline{X})=\prod\limits_{i=1}^{k-1} \gm_2^l(X_i,X_{i+1}).$$
\end{definition}

Multiple remarks are in order. The delta function above can be thought of the standard delta function $\gd_{X_1,\dots ,X_k}$ it just checks to see if all the $X_i$ are equal. The definitions for the left and right partial flag M\"obius functions make sense because of the fact that $(\Fl^k(\calP),\preceq)$ is a partially ordered set. The left partial flag M\"obius function can be equivalently defined as the element $a$ such that $a*\zeta_k =\gd_{[k]}$ (this is why we use the terminology of left and right). Similarly right partial flag M\"obius function can be equivalently defined as the element $b$ such that $\zeta_k *b =\gd_{[k]}$. Since $\gd_{[k]}$ is not a unit and $\mathcal{I}^k(\calP,R)$ is not commutative there is no reason why we should expect the functions of Definition \ref{Mob-def} to be the same. In the next example we show that for $k>2$ all the functions of Definition \ref{Mob-def} are different even for a very ``symmetric" and small poset. This is in strict contrast to the classical case $k=2$ where $\gm_2^l=\gm_2^r=\gm_2^p$. Also, note that for the product definitions case we could have used $\gm_2^r$ instead of $\gm_2^l$ since they are the same. We also remark that in the case of a poset $\calP$ with a minimal element $\hat{0}$, all of these functions are different than the ``one variable'' classical M\"obius function (by this we mean $\gm_2^l(\hat{0},\overline{X})$) on the partial flag poset $(\Fl^k(\calP),\preceq)$.

\begin{example}\label{3rank2}

Let $\calP$ be the poset in Figure \ref{3R2} (it can be thought of in many ways: the partition poset on three elements, the dim 2 braid matroid lattice of flats, the graph partition poset of the complete graph with 3 vertices, etc). Then $\mu_3^l (0,a,1)=-2$ and hence $\mu_3^l (0,1,1)=4$. But $\mu^r_3(0,1,1)=2$. Then we note that $\gm_3^l(0,a,a)=1$, $\gm_3^r(0,0,a)=1$, but $\mu_3^p(0,a,a)=-1=\mu_3^p(0,0,a)$. Also, we compute the classical M\"obius function on the partial flag poset. In this case we have $\mu_2^l((0,0,0),(0,a,a))=0$ but none of the function values $\mu_3^l(0,a,a)$, $\mu_3^r(0,a,a)$, or $\mu_3^p(0,a,a)$ are zero. This example comes up also as the counter-example to a generalized deletion-restriction formula for the higher characteristic polynomials. \begin{figure}\begin{tikzpicture}
\draw[thick] (0,0)--(1,1)--(0,2)--(0,0)--(-1,1)--(0,2);
\fill  (0,0) circle (2pt);
\fill  (0,2) circle (2pt);
\fill  (0,1) circle (2pt);
\fill  (1,1) circle (2pt);
\fill  (-1,1) circle (2pt);
\node at (-2.5,1) {$\calP=$};
\node[right] at (.1,0) {0};
\node[right] at (.1,2) {1};
\node[left] at (-1.1,1) {$a$};
\node[right] at (.1,1) {$b$};
\node[right] at (1.1,1) {$c$};

\end{tikzpicture}\caption{A poset where the left and right M\"obius functions are different.}\label{3R2}\end{figure}
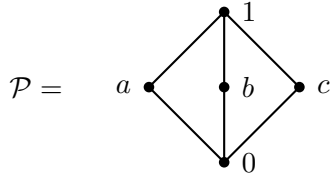

\end{example}

The fact that $\mu_k^l$, $\gm_k^r$ and $\mu^p_k$ are different seems to indicate the complexity of the invariants studied here. We choose to study $\mu_k^l$ in order to examine it's possible utility and use it in a generalization of characteristic polynomials presented in Section \ref{Char_k}. From now on we refer to $\gm_k^l$ as just $\gm_k$. Next we study $\mu_k^p$ in Subsection \ref{Mob-counting} and show that it satisfies a generalized Cross-cut theorem. For this paper use $\mu_k$ to define some fundamental invariants given by these generalizes functions. 


Using the functions $\gm_k$ and $\gz_k$ we can define multi-indexed Whitney numbers. A good reference for classical Whitney numbers is \cite{A87} and these were generalized to 2 subscripts in \cite{GZ83}.

\begin{definition}\label{whitney-numb}
Let $\calP$ be a ranked locally finite poset. Let $I=\{i_1,\dots ,i_k\}$ be an ordered $k$-tuple such that for all $j$, $i_j\in \{0,1,2,\dots, \rk \calP\}$.
\begin{enumerate}
\item The \emph{multi-indexed Whitney numbers of the first kind} are $$w_I(\calP)=\sum\limits \gm_k (X_1,X_2,\dots ,X_k)$$ where the sum is over all $k$-tuples $(X_1,\dots ,X_k)$ where $X_1\leq X_2 \leq \cdots \leq X_k$ and for all $j\in [k]$, $\rk X_j=i_j$.

\item The \emph{multi-indexed Whitney numbers of the second kind} are $$W_I(\calP)=\sum\limits \gz_k (X_1,X_2,\dots ,X_k)$$ where the sum is over all $k$-tuples $(X_1,\dots ,X_k)$ where $X_1\leq X_2 \leq \cdots \leq X_k$ and for all $j\in [k]$, $\rk X_j=i_j$.

\end{enumerate}
When the context of the poset is clear we will just write $W_I$ instead of $W_I(\calP)$.
\end{definition}

\begin{remark}

Note that the original Whitney numbers, $w_i$, of the first kind in our notation are $w_{0,i}$. The Whitney numbers of the second kind have the property that some indices are trivial in the sense that with or without these indices we have the same value. For example, if $\calP$ is a lattice then $W_{0,i}=W_i$.

\end{remark}

\begin{remark}

Let $\calP$ be a rank $n$ poset with $\hat{0}$ and $\hat{1}$. Let $[n]=\{0,1,\dots ,n\}$. We can make a function $\alpha:[n]\to \Z$ defined by $\alpha(I)=W_I$. The function $\alpha$ is called the \emph{flag f-vector} in the literature (see \cite{Karu06} and \cite{Stan94}). This function is used in the original definition of the so called \emph{cd-index} of the poset $\calP$. It would be interesting to know if there were any connection between some of the facts presented in this study and various results on the cd-index.

\end{remark}

\subsection{Boolean Lattices}\label{boolat}

Let $\B_n$ be the rank $n$ Boolean lattice. In the following lemma we compute the $k^{th}$ M\"obius function on elements in $\B_1=\{0,1\}$. 

\begin{proposition}\label{boolmob}

If $X_1\leq \cdots \leq X_k\in \B_1$ then $$\mu_k (X_1,\dots ,X_k)=\left\{ \begin{array}{ccc} (-1)^{\rk(X_1)+\cdots +\rk (X_k)} & \ \ \ & \text{if } (X_1,\dots ,X_k)\neq (1,1,\dots ,1)\\
1 & & \text{else}
\end{array}\right. .$$ 

\end{proposition}

\begin{proof}
Since as a set we can write $\B_1=\{0 ,1\}$ we have that the only possible M\"obius values are $$b_i:=\gm_k(0,\dots ,0,1,\dots ,1)$$ where there are $i$ 1's. Starting at the bottom of the flags we have that $b_0=1$ by definition. Then the indices to sum over in the recursive definition of $\gm_k$ give that  $0=b_i+b_{i-1}=b_i+(-1)^{i-1}$ as long as $i<k$. Hence for $i<k$ we have $b_i=(-1)^i$. The case $i=k$ is by definition.\end{proof}

\begin{remark} Note that for even $k$ the formula $\mu_k(\B_1)(X_1,\dots ,X_k)=(-1)^{\rk(X_1)+\cdots +\rk (X_k)}$ holds for all possible $(X_1,\dots ,X_k)$ by Proposition \ref{boolmob}. However, it does not if $k$ is odd. This is more evidence that $\gm_k$ holds some subtle information and for higher rank Boolean lattices the odd M\"obius functions $\gm_{2k+1}$ are more complicated.
\end{remark}

Combining Theorem \ref{prodI} and Proposition \ref{boolmob} we get the following.

\begin{proposition}\label{mu-bool-n}

If $\B_n$ is the rank $n$ Boolean lattice, $k$ is even and $(X_1,\dots , X_k)\in \Fl^k(\B_n)$ then $\gm_k(X_1,\dots , X_k)=(-1)^{\rk(X_1)+\cdots +\rk (X_k)}$.

\end{proposition}

Summing all these values over certain ranks can yield a result about the partial flag Whitney numbers of a Boolean lattice.

\begin{corollary}

If $k$ is even and $I=(i_1,\dots ,i_k)$ then $w_I(\B_n)=(-1)^{i_1+\cdots +i_k}W_I(\B_n)$.

\end{corollary}

Next we compute these numbers for the Boolean lattice. For any ranked poset $\calP$ we let $\calP(j)=\{X\in \calP|\rk(X)=j\}$.

\begin{proposition}\label{W(B_n)}

If $I=(i_1,\dots ,i_k)$ then $$W_I(\B_n)={n\choose i_1,i_2-i_1,i_3-i_2,\dots ,i_k-i_{k-1},n-i_k} .$$

\end{proposition}

\begin{proof}

We start at the bottom of the chain $i_1\leq \cdots \leq i_k$. There are exactly ${n\choose i_1}$ elements of rank $i_1$ in $\B_n$ (i.e. $|\B_n(j)|={n\choose j}$). Then for any $X\in \B_n(i_1)$ the restriction to $X$ is $\B_n^X\cong \B_{n-i_1}$ and the elements above $X$ of rank $i_2$ in $\B_n$ are now of rank $i_2-i_1$ in $\B_n^X$. So, for every $X\in \B_n(i_1)$ the number of elements above it is ${n-i_1\choose i_2-i_1}$. In general for every $Y\in \B_N(i_j)$ there are ${n-i_j\choose i_{j+1}-i_j}$ above it in $\B_n(i_{j+1})$. Hence $$W_I(\B_n)={n\choose i_1}{n-i_2\choose i_2-i_1}\cdots {n-i_{k-1}\choose i_k-i_{k-1}}={n\choose i_1,i_2-i_1,i_3-i_2,\dots ,i_k-i_{k-1},n-i_k}.$$\end{proof}

\subsection{A counting formula for matroids}\label{Mob-counting} One reason that the classical M\"obius function is so useful is that it can be computed many different ways. In this subsection we generalize one of these classical formulas to the setting of these partial flag M\"obius functions. There are multiple motivations for the formula we present below and we plan to use this formula to examine a relationship to a generalized Tutte polynomial in a subsequent study. Suppose that $M$ is a matroid and $L_M=L(\A)$ is its associated lattice of flats with set of atoms $\A$. For $(X_1,\dots, X_k)\in  \Fl^k(L(\A))$ set $$S(X_1,\dots, X_k)= \left\{ (B_1,\dots,B_{k-1})\in \left[2^\A\right]^{k-1}|\forall i,\ X_i\subseteq B_i \text{ and } \bigvee B_i=X_{i+1} \right\} $$ where $2^\A$ is the power set of $\A$ and $\bigvee B_i$ means the flat which is the join of the atoms in $B_i$. Now the following generalizes a special version of Rota's Cross-cut Theorem (Theorem 3 in \cite{rota64} and also see Theorem 1.1 in \cite{BlasSag-97}) where we generalize the notation in Lemma 2.35 of \cite{OT} for ease of notation.

\begin{theorem}\label{cross-cut}

For any geometric lattice $L(\A)$ with set of atoms $\A$ and all $\overline{X}=(X_1,\dots, X_k)\in \Fl^k(L(
\A))$

$$\mu_k^p(X_1,\dots, X_k)=\sum\limits_{\overline{B}\in S(\overline{X})} \prod\limits_{i=1}^{k-1}(-1)^{|B_i\backslash X_i|}=\sum\limits_{\overline{B}\in S(\overline{X})} (-1)^{\sum\limits_{i=1}^{k-1} |B_i\backslash X_i|}$$

where $\overline{B}=(B_1,\dots ,B_{k-1})$.

\end{theorem}

\begin{proof}

We do this by induction on $k$. The base case, $k=2$, is classical (see \cite{rota64}) where the cross-cut set is the set of atoms in the interval lattice $[X_1,X_2]$. Then for $k>2$ and all $\overline{X} \in \Fl^k(L(\A))$ we compute 
\begin{equation}\label{prodmob1}
\sum\limits_{\overline{B}\in S(\overline{X})} \prod\limits_{i=1}^{k-1}(-1)^{|B_i\backslash X_i|}   =\sum\limits_{\substack{B_1\supseteq X_1 \\ \bigvee B_1 =X_2}} \sum\limits_{\substack{B_2\supseteq X_2 \\ \bigvee B_2 =X_3}} \cdots \sum\limits_{\substack{B_{k-1}\supseteq X_{k-1} \\ \bigvee B_{k-1} =X_k}} \prod\limits_{i=1}^{k-1}(-1)^{|B_i\backslash X_i|}. \end{equation}
Since the last summation in (\ref{prodmob1}) only has summation variable $B_{k-1}$ we can factor out all the other $B_i$ terms for $i<k-1$ to get (\ref{prodmob1}) as
\begin{equation}\label{prodmob2} = \sum\limits_{\substack{B_1\supseteq X_1 \\ \bigvee B_1 =X_2}} \sum\limits_{\substack{B_2\supseteq X_2 \\ \bigvee B_2 =X_3}} \cdots \sum\limits_{\substack{B_{k-2}\supseteq X_{k-2} \\ \bigvee B_{k-2} =X_{k-1}}} \prod\limits_{i=1}^{k-2}(-1)^{|B_i\backslash X_i|} \left[ \sum\limits_{\substack{B_{k-1}\supseteq X_{k-1} \\ \bigvee B_{k-1} =X_k}} (-1)^{|B_{k-1}\backslash X_{k-1}|}\right] \end{equation}
Then again by Theorem .... in \cite{} (\ref{prodmob2}) becomes 
\begin{equation}\label{prodmob3} = \sum\limits_{\substack{B_1\supseteq X_1 \\ \bigvee B_1 =X_2}} \sum\limits_{\substack{B_2\supseteq X_2 \\ \bigvee B_2 =X_3}} \cdots \sum\limits_{\substack{B_{k-2}\supseteq X_{k-2} \\ \bigvee B_{k-2} =X_{k-1}}} \prod\limits_{i=1}^{k-2}(-1)^{|B_i\backslash X_i|} \mu_2 (X_{k-1},X_k) \end{equation}
Condensing (\ref{prodmob3}) and factoring out the last term (\ref{prodmob3}) becomes
\begin{equation}\label{prodmob4}
= \mu_2 (X_{k-1},X_k) \sum\limits_{(B_1,\ldots,B_{k-2})\in S(X_1,\ldots , X_{k-1})} \prod\limits_{i=1}^{k-2}(-1)^{|B_i\backslash X_i|}. 
\end{equation}
Using the induction hypothesis (\ref{prodmob4}) is now \begin{equation}\label{prodmob5} 
= \mu_2 (X_{k-1},X_k) \mu_{k-1}^p(X_1,\ldots ,X_{k-1})
\end{equation} and by definition of $\mu_{k-1}^p$ (\ref{prodmob5}) is $=\mu_2 (X_{k-1},X_k) \mu_{2}^p(X_1,X_2) \cdots \mu_2(X_{k-2},X_{k-1})$. \end{proof}

%
%
%

\section{Generalized Characteristic polynomials}\label{Char_k}

Using the Whitney numbers from Definition \ref{whitney-numb} we can define generalized characteristic polynomials.

\begin{definition}\label{defchar}

The $k^{th}$ {\emph partial flag characteristic} polynomial of a finite ranked poset $\calP$ with rank $r$ and smallest element $\hat{0}$ is 

$$\chi_k(\calP,{\bf t})=\sum\limits_{|I|=k}w_{\{\hat{0}\}\cup I}{\bf t}^{r-I}$$ 

where ${\bf t}^{r-I}=t_1^{r-i_1}t_2^{r-i_2}\dots t_k^{r-i_k}$ when $I=\{i_1,\dots, i_k\}$. For an arrangement $\A$ we define the $k^{th}$ characteristic polynomial $$\chi_k(\A,{\bf t})=\sum\limits_{(X_1,\dots ,X_k)\in \Fl^k(L(\A))}\hspace{-1cm}\mu_{k+1}(\hat{0},X_1,\dots ,X_k)t_1^{\dim X_1}\cdots t_k^{\dim X_k}.$$

\end{definition}

\begin{remark}

The usual characteristic polynomial of a poset $\calP$ is exactly $\chi_1(\calP,t)=\chi (\calP,t)$.

\end{remark}

We get the following result from viewing $\chi_k$ as an element of the incidence algebra $\mathcal{I}^k(\calP,\Z [t_1,\dots ,t_k])$ and Theorem \ref{prodI}.

\begin{proposition}\label{prodchar}

If $P$ and $Q$ are finite posets then $$\chi_k(P\times Q;t_1,\dots ,t_k)=\chi_k(P;t_1,\dots ,t_k)\chi_k(Q;t_1,\dots ,t_k).$$

\end{proposition}

Using this we can compute the characteristic polynomial of the Boolean lattices. Let $\B_n$ be the Boolean lattice of rank $n$. Again we use the product formula on $\B_n\cong (\B_1)^n$.

For $1\leq i k$ let $\tilde{w}_i=w_{0,\dots ,0,1,\dots ,1}(\B_1)$ where there are $i$ 1's. Then by Proposition \ref{boolmob} we have that $\tilde{w}_i=(-1)^i$. Thus \begin{align}\chi_k(\B_1;t_1,\dots ,t_k)&=\sum\limits_{i=0}^k(-1)^i\prod\limits_{j=1}^{k-i}t_j\label{cB_1} \\ 
&=t_1(t_2(\cdots(t_{k-1}(t_k-1)+1)\cdots +(-1)^{k-1})+(-1)^k.\label{fB_1}
\end{align} Applying Proposition \ref{prodchar} to Equations (\ref{cB_1}) and (\ref{fB_1}) we get the following proposition. 

\begin{proposition}\label{booleanform}
The characteristic polynomials of the Boolean matroids are
\begin{align*}\chi_k(\B_n;t_1,\dots ,t_k)&=\Bigg( \sum\limits_{i=0}^k(-1)^i\prod\limits_{j=1}^{k-i}t_j\Bigg)^n\\
&=\Big(t_1(t_2(\cdots(t_{k-1}(t_k-1)+1)\cdots )+(-1)^{k-1})+(-1)^k\Big)^n\\ 
\end{align*} where the last product term in the summation ($i=k$) is 1.
\end{proposition}

Using Proposition \ref{boolmob}, Proposition \ref{W(B_n)}, and the Proposition \ref{booleanform} we get a formula for multinomial coefficients.

\begin{corollary} If $k$ is even

\begin{multline*}\chi_k(\B_n;t_1,\dots ,t_k)=\Big(t_1(t_2(\cdots(t_{k-1}(t_k-1)+1)\cdots +(-1)^{k-1})+(-1)^k\Big)^n\\ 
=\sum\limits_{\{i_1,\dots ,i_k\}\subset [n] }(-1)^{i_1+\cdots +i_k}{n\choose i_1,i_2-i_1,i_3-i_2,\dots ,i_k-i_{k-1},n-i_k}t_1^{i_1}\cdots t_k^{i_k} .\end{multline*}

\end{corollary}

Next, using Proposition \ref{booleanform}, we note that $k^{th}$ generalized characteristic polynomial for the the Boolean poset satisfies a nice identity relating to the classical characteristic polynomial.

\begin{corollary}

The characteristic polynomials of the Boolean matroids satisfy

$$\chi_k(\B_n;t_1,\dots ,t_k)=(-1)^n\chi_1(\B_n;-(\chi_k(\B_1;t_1,\dots ,t_k)+(-1)^{k-1})).$$

\end{corollary}

Now we show that these higher characteristic polynomials do not satisfy a deletion-restriction formula for hyperplane arrangements. First we examine the Boolean formula to deduce what a deletion restriction formula would look like. Let $\B_n$ be the Boolean arrangement in a vector space of rank $n$ and $B_n'$ and $B_n''$ be the deletion and restriction respectively by one of the hyperplanes. Note that the deletion of the Boolean arrangement, $\B_n'$, is a Boolean arrangement of rank $n-1$ but it is just embedded in one higher dimension than needed. The restricted Boolean arrangement, $\B_n''$, is also Boolean of rank $n-1$. So, \begin{equation}\label{b'}\chi_k(\B_n';t_1,\dots ,t_k)=t_1\cdots t_k \Big(t_1(t_2(\cdots(t_{k-1}(t_k-1)+1)\cdots +(-1)^{k-1})+(-1)^k\Big)^{n-1}\end{equation} and \begin{equation}\label{b''} \chi_k(\B_n'';t_1,\dots ,t_k)=\Big(t_1(t_2(\cdots(t_{k-1}(t_k-1)+1)\cdots +(-1)^{k-1})+(-1)^k\Big)^{n-1}\end{equation}

Hence if we were to have a deletion restriction formula then it would have to be of the form \begin{equation}\label{try1} \chi_k(\A;t_1,\dots ,t_k)=\chi_k(\A';t_1,\dots ,t_k)-\Bigg( \sum\limits_{i=0}^{k-1}(-1)^i\prod\limits_{j=1}^{k-1-i}t_j\Bigg)\chi_k(\A'';t_1,\cdots ,t_k) \end{equation} where $\A'$ and $\A''$ are the deletion and restriction respectively. Note that if (\ref{try1}) were true it would generalize the usual deletion-restriction formula with $k=1$ (see \cite{OT}).

However, the next example shows that this formula is not satisfied even for $k=2$ on a rank 2 matroid. 

\begin{example}

Let $\A$ be the arrangement of 3 hyperplanes in rank 2. So, $\A$ has intersection lattice of that in Example \ref{3rank2}. Then the 2nd characteristic polynomial is $$\gc_2(\A;t_1,t_2)=t_1^2t_2^2-3t_1^2t_2+2t_1^2+3t_1t_2-6t_1+4.$$ Since the deletion $\A'$ is Boolean of rank 2 we have that $\gc_2(\A';t_1,t_2)=(t_1(t_2-1)-1)^2$. The restriction $\A''$ is $\B_1$ hence $\gc_2(\A'';t_1,t_2)=t_1(t_2-1)+1$. Now if we insert these into the formula (\ref{try1}) we get \begin{align*}\gc_2(\A';t_1,t_2)-(t_1-1)\gc_2(\A'';t_1,t_2) &=(t_1(t_2-1)+1)^2-(t_1-1)(t_1(t_2-1)+1)\\
&=t_1^2t_2^2-3t_1^2t_2+2t_1^2+3t_1t_2-4t_1+2\end{align*} which has the last two terms different than $\gc_2(\A;t_1,t_2)$.\end{example}

\begin{remark}
It seems interesting that this fails for such a simple example.  Since the characteristic polynomial satisfies the product formula of Proposition \ref{prodchar}, it would be an evaluation of the Tutte polynomial (see for example \cite{B10}). In this sense these characteristic polynomials are new invariants. However, given the complexity of the formulas even for Boolean matroids (Proposition \ref{booleanform}) the utility of these polynomials is probably low.
\end{remark}

\begin{remark}

After considering these new invariants the next step would be to understand them in the frame work of Dupont, Fink, and Moci in \cite{DFM-19}. Also, one should decide whether or not they are valuations and the set up of Ardila and Sanchez in \cite{Ardila-Sanchez} seems particularly fitting.

\end{remark}

\bibliographystyle{amsplain}


\providecommand{\bysame}{\leavevmode\hbox to3em{\hrulefill}\thinspace}
\providecommand{\MR}{\relax\ifhmode\unskip\space\fi MR }
\providecommand{\MRhref}[2]{%
  \href{http://www.ams.org/mathscinet-getitem?mr=#1}{#2}
}
\providecommand{\href}[2]{#2}

\end{document}